\documentclass[12pt]{amsart}
\usepackage{amssymb}
\usepackage{hyperref}

\newtheorem{thm}{Theorem}[section]
\newtheorem{lem}[thm]{Lemma}

\newtheorem{prop}[thm]{Proposition}
\newtheorem{ex}[thm]{Example}

\newtheorem{que}[thm]{Question}
\newtheorem*{prob*}{Open problem}

\theoremstyle{definition}
\newtheorem{defi}[thm]{Definition}

\theoremstyle{remark}
\newtheorem{rem}[thm]{Remark}
\newtheorem*{rem*}{Remark}

\DeclareMathOperator{\s}{span}

\newcommand{\kringel}{\mathbin{\raise1pt\hbox{$\scriptstyle\circ$}}}
\newcommand{\pkt}{\mathbin{\raise0pt\hbox{$\scriptstyle\bullet$}}}

\newcommand{\C}{\mathbb{C}}

\newcommand{\F}{\mathbb{F}}
\newcommand{\N}{\mathbb{N}}

\newcommand{\Q}{\mathbb{Q}}

\newcommand{\Z}{\mathbb{Z}}

\newcommand{\diag}{\mathop{\rm diag}}

\newcommand{\La}{\mathfrak{a}}
\newcommand{\Lb}{\mathfrak{b}}

\newcommand{\Lf}{\mathfrak{f}}
\newcommand{\Lg}{\mathfrak{g}}
\newcommand{\Lh}{\mathfrak{h}}

\newcommand{\Ln}{\mathfrak{n}}

\newcommand{\CN}{\mathcal{N}}

\newcommand{\al}{\alpha}
\newcommand{\be}{\beta}

\newcommand{\ra}{\rightarrow}

\renewcommand{\phi}{\varphi}

\begin{document}

% Ab hier duerfen Sie wieder.

\title[Derived Length]{Derived length and nildecomposable Lie algebras}
%  Die Kurzfassung kommt oben ueber die Seiten, sie steht in eckigen Klammern
%  Auch Autorennamen koennen eine Kurzfassung haben

\author[D. Burde]{Dietrich Burde}
\address{Fakult\"at f\"ur Mathematik\\
Universit\"at Wien\\
  Nordbergstr. 15\\
  1090 Wien \\
  Austria}
\email{dietrich.burde@univie.ac.at}

\date{\today}

\thanks{The author is supported by the FWF, Project P21683.}

\begin{abstract}
We study the minimal dimension of solvable and nilpotent Lie algebras
over a field of characteristic zero with given derived length $k$. 
This is motivated by questions on nildecomposable Lie algebras $\Lg=\La+\Lb$,
arising in the context of post-Lie algebra structures. The question is, how the derived 
length of $\Lg$ can be estimated in terms of the derived length and nilpotency classes 
of the two nilpotent subalgebras $\La$ and $\Lb$. 
\end{abstract}

\maketitle

\section{Introduction}

What is the minimal dimension of a finite-dimenisonal nilpotent Lie algebra over a field of 
characteristic zero with given derived length $k$ ? This question has been studied by various 
authors. In general, there exist only lower and upper bounds for this dimension. We review 
the known results and provide a new lower bound for the case $k=4$ by constructing a rational
nilpotent Lie algebra of dimension $14$ with derived length $k=4$. The corresponding question 
for finite $p$-groups has a long history, going back to work of Burnside in $1913$. Our 
motivation to study this minimal dimension comes from open questions about nildecomposable Lie algebras.
A finite dimensional Lie algebra is called {\it nildecomposable}, if it is the sum of two 
nilpotent subalgebras. Although it is very interesting to consider also infinite-dimensional Lie algebras,
we assume here all Lie algebras to be finite-dimensional. We write $\Lg=\La+\Lb$. The sum is understood as 
a vector space sum and need not be direct. It is a well known result that a nildecomposable Lie algebra over 
a field of characteristic zero is solvable. This was first proved by Goto in \cite{GOT}. 
It is easy to see that we may assume that the field is algebraically closed. In characteristic 
$p$ the situation is as usual more complicated. A nildecomposable Lie algebra is solvable 
over a field of characteristic $p$ with $p\ge 5$, see \cite{ZUS} and the references therein. 
There is a counterexample in characteristic $2$, see \cite{PE1}.  
Historically the question was first asked about nildecomposable groups $G=AB$.
There are numerous results for such groups. We will give a short survey in section $6$. \\[0.2cm]
Decompositions of the form $\Lg=\La+\Lb$ for Lie algebras arise also in geometry.
We have recently obtained some results on the existence of so called post-Lie
algebra structures on pairs of Lie algebras $(\Lg,\Ln)$, see \cite{BU33}, \cite{BU41}, \cite{BU44} for
details and the geometric background. One conclusion was, that if $\Lg$ is nilpotent, such a structure 
can only exist, if $\Ln$ is solvable. The proof proceeded to show that $\Ln$ must be nildecomposable,
hence solvable. \\[0.2cm]
A very subtle question is about the {\it derived length} of a nildecomposable Lie algebra.
If $\Lg=\La+\Lb$, then it is believed that the derived length $d(\Lg)$ of $\Lg$
is bounded by a linear function in the nilpotency classes $c(\La)$ and $c(\Lb)$
of $\La$ and $\Lb$. The question is even, if it is possible to find a counterexample
to the estimate
\[
d(\Lg) \le c(\La) + c(\Lb),
\]
or not. If $\La$ and $\Lb$ are abelian, then the above estimate is true, since then $d(\Lg) \le 2$ 
by It\^{o}'s theorem. The investigate the cases $d(\Lg) \ge 3$ we need to study the minimal dimension 
of solvable and nilpotent Lie algebras with given derived length $k$.

\section{Preliminaries}

Let $\Lg$ be a finite-dimensional Lie algebra over an arbitrary field $K$.
The {\it lower central series} of $\Lg$ is defined by 
\[
\Lg^0=\Lg \supseteq \Lg^1 \supseteq \Lg^2 \supseteq \Lg^3  \supseteq \cdots
\]
where the ideals $\Lg^i$ are recursively defined by $\Lg^i=[\Lg,\Lg^{i-1}]$
for all $i\ge 1$. Thus we have $\Lg^0=\Lg$, $\Lg^1=[\Lg,\Lg]$, $\Lg^2=[\Lg,[\Lg,\Lg]]$,  
and so on. The {\it derived series} of $\Lg$ is defined by
\[
\Lg^{(0)}=\Lg \supseteq \Lg^{(1)} \supseteq \Lg^{(2)} \supseteq \Lg^{(3)}  \supseteq \cdots
\]
where the ideals $\Lg^{(i)}$ are recursively defined by $\Lg^{(i)}=[\Lg^{(i-1)},\Lg^{(i-1)}]$
for all $i\ge 1$. We have  $\Lg^{(0)}=\Lg$, $\Lg^{(1)}=[\Lg,\Lg]$, $\Lg^{(2)}=[[\Lg,\Lg], [\Lg,\Lg]]$,  
and so on. A Lie algebra $\Lg$ is called {\it nilpotent}, if $\Lg^k=0$ for some $k$. It is called
{\it solvable}, if $\Lg^{(k)}=0$ for some $k$. Every nilpotent Lie algebra is also solvable, and we
have the inclusion
\[
\Lg^{(k)} \subseteq \Lg^{2^k-1} \text{ for all }k\ge 1.
\]

\begin{defi}
A Lie algebra $\Lg$ of dimension $n\ge 1$ is called $k$-step nilpotent if $\Lg^k=0$, but 
$\Lg^{k-1}\neq 0$ for some $k\ge 1$. In this case we call $k$ the {\it nilpotency class} of $\Lg$, and
denote it by $c(\Lg)$. If $c(\Lg)=n-1$, then $\Lg$ is called {\it filiform nilpotent}.
\end{defi}

Denote by $\Lf$ the Lie algebra with basis $\{ e_1,\ldots ,e_n\}$, where the only non-trivial
Lie brackets are given by $[e_1,e_i]=-[e_i,e_1]=e_{i+1}$ for $2\le i\le n-1$. Writing
$\Lf_k=\s \{e_k,\ldots ,e_n \}$ we have $\Lf^1=\Lf_3$, $\Lf^2=\Lf_4$, $\ldots$ , $\Lf^{n-2}=\Lf_n=\s \{e_n \}$ 
and $\Lf^{n-1}=0$. Hence $c(\Lf)=n-1$ and $\Lf$ is filiform nilpotent. This algebra is called the standard 
graded filiform nilpotent Lie algebra of dimension $n$.

\begin{defi}
A Lie algebra $\Lg$ of dimension $n\ge 1$ is called $k$-step solvable if $\Lg^{(k)}=0$, but 
$\Lg^{(k-1)}\neq 0$ for some $k\ge 1$. In this case we call $k$ the {\it solvability class} of $\Lg$, and
denote it by $d(\Lg)$. It is also called the {\it derived length} of $\Lg$.
\end{defi}

For example, the filiform nilpotent Lie algebra $\Lf$ from above satisfies $d(\Lf)=2$, since
\begin{align*}
\Lf^{(1)} & = \Lf_3\neq 0, \\
\Lf^{(2)} & = [\Lf^{(1)},\Lf^{(1)}]=[\Lf_3,\Lf_3]=0.
\end{align*}

\begin{defi}
Let $k\ge 1$ be a positive integer. Denote by $\al (k)$ the 
minimal dimension of a {\it nilpotent} Lie algebra over a field $K$ of characteristic zero with derived 
length $d(\Lg)=k$. Denote by $\be (k)$ the minimal dimension of a {\it solvable} Lie algebra over a 
field $K$ of characteristic zero with derived length $d(\Lg)=k$.
\end{defi}

\begin{rem}
One can also define $\al(k)$ and $\be(k)$ by allowing an arbitrary field $K$. Note that this makes 
a difference. In fact, in prime characteristic sometimes exist new Lie algebras of smaller dimension 
with derived length $k$, which have no counterpart in characteristic zero. An example is discussed
in remark $\ref{3.4}$. 
\end{rem}

\begin{defi}
An algebra $(A,\cdot)$ over a field $K$ with bilinear product $(x,y) \mapsto x\cdot y$
is called a {\it left-symmetric algebra}, if the product is left-symmetric, i.e., if the identity
\begin{equation*}
x\cdot (y\cdot z)-(x\cdot y)\cdot z= y\cdot (x\cdot z)-(y\cdot x)\cdot z
\end{equation*}
is satisfied for all $x,y,z \in A$. 
\end{defi}

A left-symmetric algebra is sometimes called pre-Lie algebra. In fact, the
commutator $[x,y]:=x\cdot y-y\cdot x$ defines a Lie bracket. An associative
algebra is a special case of a left-symmetric algebra.

\begin{defi}
An {\it affine structure} on a Lie algebra $\Lg$ over $K$ is a left-symmetric product 
$\Lg \times \Lg \rightarrow \Lg$ satisfying $[x,y]=x\cdot y-y\cdot x$ for all $x,y \in \Lg$.
\end{defi}

Affine structures on Lie algebras and left-symmetric algebras have been studied intensively.
For a survey see \cite{BU24}. We recall the following result.

\begin{prop}\label{2.7}
Let $\Lg$ be a finite-dimensional Lie algebra over a field $K$ of characteristic zero.
Assume that $\Lg$ admits a non-singular derivation. Then $\Lg$ is nilpotent and admits
an affine structure.
\end{prop}

\section{Derived length of nilpotent Lie algebras}

What can we say about $\al(k)$, the minimal dimension of a nilpotent Lie algebra 
over a field $K$ of characteristic zero with derived length $k$ ? 
In general it is impossible to determine this dimension explicitly. However, we have
lower and upper bounds for $\al(k)$. These bounds, and here in particular the lower bounds,
have been discussed in the literature by various authors. Our main reference is the paper of 
Bokut \cite{BOK}, which cites previous work by Hall, Dixmier and Patterson \cite{PAT}. 
See also \cite{BRS} for more results on $2$-generated Lie algebras. \\
It is not difficult for us to provide an upper bound of the form $\al(k)\le 2^k-1$ for all $k\ge 2$.

\begin{prop}\label{3.1}
For any $k\ge 2$ there is a $\N$-graded filiform nilpotent Lie algebra $\Lf$, defined over $\Q$ 
of dimension $2^k-1$ with $d(\Lf)=k$. This Lie algebra admits an non-singular derivation and
hence an affine structure.
\end{prop}

\begin{proof}
In \cite{BU32} we have constructed such Lie algebras, even with an adapted basis, i.e., 
with a basis $(e_1,\ldots ,e_n)$ such that $[e_1,e_j] = e_{j+1}$ for all 
$2\le j\le n-1$. Let $n\ge 3$. The filiform nilpotent Lie algebra $\Lf_{\frac{9}{10},n}$ of dimension $n$ by
is given as follows:
\begin{align*}
[e_1,e_j] & = e_{j+1}, \quad 2\le j\le n-1,\\[0.2cm]
[e_i,e_j] & = \frac{6(j-i)}{j(j-1)\binom{j+i-2}{i-2}}e_{i+j},
\quad 2\le i \le j; \; \; i+j \leq n. 
\end{align*}
Note that the Jacobi identity holds, see \cite{BU32}. Hence $\Lf_{\frac{9}{10},n}$ is a Lie algebra.
It is generated by $e_1$ and $e_2$. For $n\ge 7$ these generators satisfy, among the relations given above,
$[e_1,e_2]=e_3$, $[e_2,e_3]=e_5$ and  $[e_2,e_5]=\frac{9}{10}e_7$. We write $\Lf_k=\s \{e_k,\ldots ,e_n \}$
as before. For $n=2^k-1$ this Lie algebra has derived length equal to $k$, since we have
\[
\Lg^{(0)} =\Lg, \; \Lg^{(1)}=\Lf_3, \;  \Lg^{(2)}=\Lf_7, \ldots , 
\Lg^{(k-1)} =\Lf_{2^{k}-1}=\s \{e_{2^{k}-1}\}, \; \Lg^{(k)} =0. 
\]
The Lie algebras $\Lf_{\frac{9}{10},n}$ are, of course, graded by positive integers. 
In the given adapted basis, $D=\diag(1,2,\ldots, n)$ is an invertible derivation.
Hence there exists an affine structure by proposition $\ref{2.7}$.
\end{proof}

\begin{ex}
For $k=4$, the Lie algebra $\Lf_{\frac{9}{10},15}$ has the following Lie brackets. We have
$[e_1,e_j] = e_{j+1}$ for $2\le j\le 14$, and

$$
\begin{array}{lll}
[e_2,e_3]=e_5,               & [e_3,e_5]=\frac{1}{10}e_8,  & [e_4,e_9]=\frac{1}{132}e_{13}, \\[0.2cm]
\left[e_2,e_4\right]=e_6,    & [e_3,e_6]=\frac{3}{35}e_9,  & [e_4,e_{10}]=\frac{1}{165}e_{14}, \\[0.2cm]
\left[e_2,e_5\right]=\frac{9}{10}e_7, & [e_3,e_7]=\frac{1}{14}e_{10},& [e_4,e_{11}]=\frac{7}{1430}e_{15}, \\[0.2cm]
\left[e_2,e_6\right]=\frac{4}{5}e_8, & [e_3,e_8]=\frac{5}{84}e_{11}, & [e_5,e_6]=\frac{1}{420}e_{11}, \\[0.2cm]
\left[e_2,e_7\right]=\frac{5}{7}e_9, & [e_3,e_9]=\frac{1}{20}e_{12}, & [e_5,e_7]=\frac{1}{420}e_{12}, \\[0.2cm]
\left[e_2,e_8\right]=\frac{9}{14}e_{10},& [e_3,e_{10}]=\frac{7}{165}e_{13},&[e_5,e_8]=\frac{3}{1540}e_{13},\\[0.2cm]
\left[e_2,e_9\right]=\frac{7}{12}e_{11},& [e_3,e_{11}]=\frac{2}{55}e_{14}, & [e_5,e_9]=\frac{1}{660}e_{14},\\[0.2cm]
\left[e_2,e_{10}\right]=\frac{8}{15}e_{12},&[e_3,e_{12}]=\frac{9}{286}e_{15},&[e_5,e_{10}]=\frac{1}{858}e_{15},
\\[0.2cm]
\left[e_2,e_{11}\right]=\frac{27}{55}e_{13},&[e_4,e_5]=\frac{1}{70}e_9,&[e_6,e_7]=\frac{1}{2310}e_{13},\\[0.2cm]
\left[e_2,e_{12}\right]=\frac{5}{11}e_{14},&[e_4,e_6]=\frac{1}{70}e_{10},&[e_6,e_8]=\frac{1}{2310}e_{14}, \\[0.2cm]
\left[e_2,e_{13}\right]=\frac{11}{26}e_{15},&[e_4,e_7]=\frac{1}{84}e_{11},&[e_6,e_9]=\frac{1}{2860}e_{15},\\[0.2cm]
\left[e_3,e_4\right]=\frac{1}{10}e_7,&[e_4,e_8]=\frac{1}{105}e_{12}, & [e_7,e_8]=\frac{1}{12012}e_{15}.\\[0.2cm]
\end{array}
$$
\end{ex}

\begin{rem}
The estimate $\al(k)\le 2^k-1$ also follows from the results in \cite{PAN}.
However, it is more convenient for us to use the Lie algebras $\Lf_{\frac{9}{10},n}$.
\end{rem}

Concerning lower bounds there is the following result of Bokut, see \cite{BOK}.

\begin{prop}\label{3.3}
Let $\Lg$ be a nilpotent Lie algebra over a field of characteristic zero with
derived length $k\ge 4$. Then
\[
\dim (\Lg) \ge 2^{k-1}+2k-3.
\]
\end{prop}

\begin{rem}\label{3.4}
In \cite{BOK} it is shown that the estimate is also true for nilpotent Lie algebras over a 
field of prime characteristic $p$ with $p\ge 5$. For $p=2$ Bokut gives a counterexample, with $k=4$. 
In fact, the following brackets define a nilpotent Lie algebra of dimension
$12$ with nilpotency class $10$ and derived length $k=4$ over a field of characteristic $2$.

$$
\begin{array}{lll}
[e_1,e_2]=e_4,            & [e_3,e_4]=e_6,  & [e_4,e_{9}]=e_{11}, \\[0.2cm]
\left[e_1,e_3\right]=e_5, & [e_3,e_7]=e_8,  & [e_5,e_6]=e_{8}, \\[0.2cm]
\left[e_1,e_6\right]=e_7, & [e_3,e_{9}]=e_{10},& [e_5,e_7]=e_{9}, \\[0.2cm]
\left[e_1,e_8\right]=e_9, & [e_3,e_{11}]=e_{12}, & [e_5,e_{8}]=e_{10}, \\[0.2cm]
\left[e_1,e_{10}\right]=e_{11}, & [e_4,e_5]=e_7, & [e_5,e_{9}]=e_{11}, \\[0.2cm]
\left[e_2,e_5\right]=e_{6},& [e_4,e_6]=e_8, &[e_5,e_{10}]=e_{12},\\[0.2cm]
\left[e_2,e_7\right]=e_{8},& [e_4,e_7]=e_9, & [e_6,e_9]=e_{12},\\[0.2cm]
\left[e_2,e_{9}\right]=e_{10},&[e_4,e_8]=e_{10},&[e_7,e_8]=e_{12}.\\[0.2cm]
\end{array}
$$
\end{rem}

Bokut's result does not apply for $k\le 3$. But in this case we have sharp results
for $\al(k)$, since we have a classification of nilpotent Lie algebras in low dimensions.

\begin{prop}\label{3.5}
Every nilpotent Lie algebra of dimension $n\le 8$ has derived length $k\le 3$.
We have $\al(1)=1$, $\al(2)=3$ and $\al(3)=6$.
\end{prop}

\begin{proof}
Let $\Lg$ be a nilpotent Lie algebra of dimension $n\le 8$. Then its nilpotency class
satisfies $c(\Lg)\le 7$, so that $\Lg^7=0$, and $\Lg^{(3)}\subseteq \Lg^7=0$.
Hence $\Lg$ is at most of solvability class $k\le 3$. We can compute $\al(k)$ easily.
The case $k=1$ is trivial. We assume that $\dim (\Lg)\ge 1$ for a nilpotent Lie algebra.
For $k=2$ the classification shows that the standard graded filiform Lie algebra of
dimension $3$, which is the Heisenberg Lie algebra, is the Lie algebra of minimal
dimension with derived length $k=2$. In dimension two, every nilpotent Lie algebra
is abelian, over any field $K$. Again from the classification we see that all nilpotent
Lie algebras of dimension $5$ have derived length $k\le 2$. On the other hand, there
are filiform nilpotent Lie algebras of dimension $6$ with derived length $k=3$.
Here is an example, defined over $\Q$:
\[
\begin{array}{ll}
[e_1,e_i]=e_{i+1}, \; 2\le i\le 5; & [e_2,e_5]=e_6,\\
\left[e_3,e_4\right]=-e_6. & 
\end{array}
\]
This means $\al(3)=6$. 
\end{proof}

It is very natural to consider next the case $k=4$. By proposition $\ref{3.1}$ and $\ref{3.3}$
we have $13\le \al(4)\le 15$. Surprisingly, nothing more seems to be known so far.
The filiform Lie algebra $\Lf_{\frac{9}{10},15}$ of dimension $15$ has derived length $4$. Can we do
better ? For derived length $k\le 3$ there was no better choice than a filiform nilpotent
Lie algebra. It turns out that the Lie algebra $\Lf_{\frac{9}{10},15}$ is optimal among filiform Lie algebras. 
In other words, the minimal dimension of a filiform nilpotent Lie algebra in characteristic zero 
with derived length $k=4$ is equal to $15$.

\begin{prop}\label{3.6}
Let $\Lf$ be a filiform nilpotent Lie algebra of dimension $n\le 14$ over a field $K$ of characteristic
zero. Then $d(\Lf)\le 3$.
\end{prop}

\begin{proof}
It is well known that for any filiform Lie algebra $\Lf$ there exists a so called {\it adapted basis} 
$(e_1,\ldots , e_n)$ such that
\begin{align*}
[e_1,e_i] & = e_{i+1}, \;\; i=2,\ldots , n-1 \\
[e_i,e_j] & \in \Lf_{i+j}, \;\; i,j \ge 2,\, i+j \le n \\
[e_{i+1},e_{n-i}] & = (-1)^i\al e_n , \;\;  1\le i <n-1 
\end{align*}
with a certain scalar $\al$, which is automatically zero if $n$ is odd,
and  $\Lf_k=\s \{ e_k,\ldots ,e_n \}$. Here the undefined brackets are zero, and we
set $e_k=0$ for $k>n$. For $n\le 13$ it follows that $\Lf^{(1)}=\Lf_3$, and $\Lf^{(2)}\subseteq \Lf_7$ is
abelian. Hence $\Lf^{(3)}=0$. For $n=14$ we have to do more. Then we write the
Lie brackets with respect to the adapted basis with parameters $\al_{k,s}$, where
$(k,s)$ is contained in the index set
\[
\{(k,s)\in \N \times \N \mid 2 \le k \le [n/2],\,2k+1 \le s \le n \}\cup \{([n/2],n)\}.
\]
These parameters $\al_{k,s}$ have to satisfy certain polynomial equations imposed by the 
Jacobi identity. Furthermore all brackets $[e_i,e_j]$ are completely determined by the brackets
$[e_1,e_i] = e_{i+1}$ for $2\le i\le n-1$ and 
\begin{align*}
[e_k,e_{k+1}] & = \sum_{s=2k+1}^{n}\al_{k,s}e_s , \; 2\le k \le [n/2].
\end{align*}
In particular we have $[e_7,e_8]=\al_{7,14}e_{15}$ and $[e_i,e_j]=0$ for all $i+j\ge 16$.
If $\al_{7,14}=0$, then $\Lf^{(3)}=[\Lf^{(2)},\Lf^{(2)}]=[\Lf_7,\Lf_7]=0$ just as before. 
Hence we may assume that $\al_{7,14}\neq 0$. Now the Jacobi identity applied to the basis vectors
$\{e_4,e_5,e_6\}$, $\{e_2,e_6,e_7\}$, $\{e_2,e_3,e_4\}$, $\{e_2,e_3,e_8\}$ and $\{e_3,e_4,e_5\}$
implies the following equations in the parameters $(x_1,x_2,x_3)=(\al_{4,9},\al_{3,7},\al_{2,5})$:
\begin{align*}
0 & = x_1 (x_2+2x_3),\\
0 & = 3x_1(x_2+2x_3-21x_1)+2(3x_2-x_3)(x_2-x_3),\\
0 & =x_1(5x_1+3x_2).
\end{align*}
The only solution in characteristic zero is $(x_1,x_2,x_3)=(0,0,0)$. This means $\Lf^{(2)}\subseteq \Lf_8$
and $\Lf^{(3)}=[\Lf_8,\Lf_8]=0$.
\end{proof}

In view of this result we have to look for nilpotent Lie algebras $\Lg$ of derived length $k=4$,
which are {\it not} filiform. So let us assume that $\Lg^{(4)}=0$ and $\Lg^{(3)}\neq 0$. It seems very 
useful to follow the ideas in \cite{BOK}. There we start with the equation
\[
\dim (\Lg)=\dim \left( \Lg/\Lg^{(1)}\right)+\dim \left( \Lg^{(1)}/\Lg^{(2)}\right)+
\dim \left( \Lg^{(2)}/\Lg^{(3)}\right)+ \dim \left( \Lg^{(3)}\right).
\]
Then the terms on the right-hand side are estimated under various assumptions. If we assume that
our Lie algebra is generated by $2$ elements, then it follows from  \cite{BOK} that
\begin{align*}
d_1 & =\dim \left( \Lg/\Lg^{(1)}\right) \ge 2, \; d_2 =\dim \left( \Lg^{(1)}/\Lg^{(2)}\right) \ge 4,\\
d_3 & =\dim \left( \Lg^{(2)}/\Lg^{(3)}\right) \ge 6, \; d_4 =\dim \left( \Lg^{(3)}\right) \ge 1.
\end{align*}
To construct such a Lie algebra of minimal dimension, one should minimize
these dimensions. I was not able to find a nilpotent Lie algebra
with $2$ generators and $(d_1,d_2,d_3,d_4)=(2,4,6,1)$. However, the calculations
showed that there is indeed a rational nilpotent Lie algebra with two generators and
$(d_1,d_2,d_3,d_4)=(2,5,6,1)$. It has dimension $14$.

\begin{prop}\label{3.7}
There exists a rational, nilpotent Lie algebra $\Lg$ of dimension $14$ and derived length
$k=4$. It has nilpotency class $c(\Lg)=11$ and is graded by positive integers. 
\end{prop}

\begin{proof}
The idea is to search for a rational, nilpotent Lie algebra $\Lg$ with two generators
$e_1$ and $e_2$ of dimension $14$ satisfying $(d_1,d_2,d_3,d_4)=(2,5,6,1)$ and $c(\Lg)\le 12$.
See \cite{ENS} for the group case. After some computation we found the following Lie algebra:
$$
\begin{array}{lll}
[e_1,e_2]=e_3,            & [e_2,e_7]=2e_8,  & [e_3,e_{12}]=-e_{14}, \\[0.2cm]
\left[e_1,e_3\right]=e_4, & [e_2,e_8]=2e_9,  & [e_4,e_5]=-3e_{10}, \\[0.2cm]
\left[e_1,e_4\right]=e_5, & [e_2,e_{10}]=e_{11},& [e_4,e_6]=-3e_{11}, \\[0.2cm]
\left[e_1,e_5\right]=e_7, & [e_2,e_{13}]=e_{14}, & [e_4,e_{10}]=e_{13}, \\[0.2cm]
\left[e_1,e_6\right]=e_8, & [e_3,e_4]=-e_6, & [e_4,e_{11}]=e_{14}, \\[0.2cm]
\left[e_1,e_8\right]=e_{10},& [e_3,e_5]=-e_8, &[e_5,e_6]=-3e_{12},\\[0.2cm]
\left[e_1,e_9\right]=e_{11},& [e_3,e_6]=-2e_9, & [e_5,e_8]=-e_{13},\\[0.2cm]
\left[e_1,e_{11}\right]=e_{12},&[e_3,e_7]=2e_{10},&[e_5,e_9]=-e_{14},\\[0.2cm]
\left[e_1,e_{12}\right]=e_{13},&[e_3,e_8]=e_{11}, &[e_6,e_7]=2e_{13},\\[0.2cm]
\left[e_2,e_5\right]=e_6, &[e_3,e_{10}]=e_{12},&[e_6,e_8]=e_{14}. \\[0.2cm]
\end{array}
$$
This Lie algebra has integer coefficients and satisfies $d(\Lg)=4$ and $c(\Lg)=11$.
More precisely we have 
\begin{align*}
\left(\dim (\Lg^{(1)}),\ldots , \dim (\Lg^{(3)})\right) & = (12,7,1), \\
\left( \dim (\Lg^1),\ldots ,  \dim (\Lg^{10})\right) & = (12,11,10,9,7,6,4,3,2,1).
\end{align*}

The Lie algebra is graded by positive integers, i.e., has invertible derivations. Examples 
are given by
\begin{align*}
D & =\diag(\al,\be,\al+\be,2\al+\be,3\al+\be,3\al+2\be,4\al+\be, 4\al+2\be, \\
  & \hspace{1.5cm}  4\al+3\be,5\al+2\be,5\al+3\be,6\al+3\be,7\al+3\be,7\al+4\be)
\end{align*}
for suitable $\al,\be \in \Q$.
\end{proof}

\begin{rem}
The results in \cite{BOK} seem to indicate that there is no rational nilpotent Lie algebra
of dimension $13$ with derived length $k=4$. Hence we should have $\al(4)=14$. On the other
hand, this is far from obvious.
\end{rem}

We note an interesting consequence for the degeneration theory of algebras,
which seems not to be in the literature. For background on degenerations see \cite{BU36}.

\begin{prop}
In the variety $\CN_{14}(\C)$ of complex, nilpotent Lie algebras of dimension $14$
there exists a nilpotent Lie algebra of nilpotency class $11$ which is
not a degeneration of any filiform nilpotent Lie algebra of dimension $14$.
\end{prop}

\begin{proof}
Denote by $\Lh$ the nilpotent Lie algebra of proposition $\ref{3.7}$ over the
complex numbers. We have $d(\Lh)=4$ and $\dim (\Lh)=14$. Assume that there is a degeneration
$\Lg \ra_{\rm deg}\Lh$ with $\Lg$ filiform nilpotent of dimension $14$.
Then $d(\Lg)\ge d(\Lh)$, see for example \cite{BU09}, Lemma $1$. By proposition 
$\ref{3.6}$ we have $d(\Lg)\le 3$.
This is a contradiction.
\end{proof}

\section{Derived length of finite {\rm p}-groups}

Historically, the question about the derived length is much older for nilpotent groups than
for nilpotent Lie algebras. Indeed, a classical question about finite groups of prime power order, 
so called $p$-groups, is the following, see \cite{BUR1}.

\begin{que}[Burnside 1913]
What is the smallest order of a group with prime power order and derived length $k$ ?
\end{que}

This order is denoted by $p^{\be_p(k)}$. For $1\le k\le 3$ the answer is well known.
A non-abelian $p$-group has order at least $p^3$, so that $\be_p(1)=2$ for all primes $p$.
For $p\ge 5$ we have $\be_p(2)=6$, whereas $\be_2(2)=\be_3(2)=7$, see \cite{B1}.
In general, there are only lower and upper bounds for $k\ge 4$. In \cite{MAN} we find 
interesting results and references on the problem, including the following lower bound for $\be_p (k)$:

\begin{prop}
For $k\ge 4$ and all primes $p$ we have
\[
\be_p (k) \ge 2^{k-1}+2k-4.
\]
\end{prop}

The result is very similar to the result in proposition $\ref{3.3}$.
There is a strong link between the results for $p$-groups and graded Lie rings over
$\F_p$. But one must also be cautious however, as the derived series of the graded Lie ring
does not necessarily correspond to the derived series of the group. Furthermore, it is not
clear how to pass from graded Lie rings over $\F_p$ to possibly non-graded Lie algebras
over a field of characteristic zero. \\
For $k=4$ the above bound is not sharp and more is known. We have $\be_p(4)=14$ for all 
primes $p\ge 3$, and for $p=2$ we have $\be_2(4)=13$, see \cite{ERI}. 
The prime $p=2$ always causes some problems. The proofs use indeed often results on graded Lie rings 
of $p$-power order. Denote the smallest order of a graded Lie ring with $p$-power order and derived length $k$
by $p^{\al_p(k)}$. To show that $\be_p(4)=14$ for all $p\ge 5$ it was first shown that $\al_p(4)\ge 14$ for 
all primes $p\ge 5$, see \cite{ERI}.

\section{Derived length and dimension of solvable Lie algebras}

What can we say about $\be(k)$, the minimal dimension of a solvable Lie algebra 
over a field $K$ of characteristic zero with derived length $k$ ?
Clearly we have $\be(k)\le \al(k)$. Moreover, if $\Lg$ is a solvable Lie algebra of dimension $n$
over a field of characteristic zero, then $\Lg^{(1)}$ is nilpotent of dimension at most $n-1$.
Hence we obtain also bounds for $\be(k)$. For detailed results see \cite{PAT}. In
particular it is known that $\be(1)=1$, $\be(2)=2$, $\be(3)=4$ and $\be(4)=7$.
We want to point out the following fact.

\begin{lem}
Let $\Ln$ be a nilpotent Lie algebra of dimension $n$ with $d(\Ln)=k$.
Suppose that $\Ln$ admits a non-singular derivation $D$. Then $\Lg=\Ln\rtimes \langle D \rangle$
is a solvable Lie algebra of dimension $n+1$ with $d(\Lg)=k+1$.
\end{lem}

\begin{proof}
The Lie bracket of $\Lg$ is given by
\[
[(x,rD),(y,sD)]=([x,y]_{\Ln}+rD(y)-sD(x),0)
\] 
for $r,s\in K$ and $x,y\in \Ln$. Since $D$ is non-singular, we have
$\Ln=D(\Ln)\subseteq [\Lg,\Lg]\subseteq \Ln$, so that $\Ln=[\Lg,\Lg]$ and $\Lg^{(i+1)}=\Ln^{(i)}$
for all $i$. This finishes the proof.
\end{proof}

Unfortunately most nilpotent Lie algebras of dimension $n\ge 7$ do not admit a non-singular derivation.
On the other hand, the Lie algebras of proposition $\ref{3.1}$ and $\ref{3.7}$ 
do admit such a derivation. Hence we obtain the following results.

\begin{prop}\label{5.2}
For any $k\ge 3$ there is a solvable rational Lie algebra $\Lg$ of dimension $2^{k-1}$ with
$d(\Lg)=k$. Its nilradical is filiform nilpotent.
\end{prop}

\begin{prop}
There exists a rational, solvable Lie algebra $\Lg$ of dimension $15$ and derived length $k=5$.  
\end{prop}

It is interesting to note that the existence of a regular automorphism of finite order for
$\Lg$ gives a strong condition on the derived length of $\Lg$, see \cite{STW}:

\begin{prop}
Suppose that  $\Lg$ is a Lie algebra admitting a regular automorphism of finite period $n\ge 4$.
Then $\Lg$ is solvable with $d(\Lg)\le 2^{n-4}+\lfloor \log_2(n-1)\rfloor$. If $n=p\ge 5$
is a prime number then $\Lg$ is nilpotent with $d(\Lg)\le 2^{p-5}+\lfloor \log_2(p-3)\rfloor$.
\end{prop}

\section{Nildecomposable groups}

Let $G=AB$ be a product of two subgroups. It is called a factorization of $G$.
The product $AB$ is a subgroup if and only if $AB=BA$. Therefore, $A$ and $B$ are 
permutable in this factorization. There are many results on factorizations of groups.
We refer to \cite{MAN} and the references cited within. The Wielandt-Kegel theorem states that 
if $G$ is {\it finite}, and $A$ and $B$ are nilpotent subgroups, then $G$ is solvable.
Such solvable groups are called nildecomposable, or dinilpotent. 
However, it is not known if the derived length of $G$ can be bounded by $c(A)$ and $c(B)$. 
The first conjecture in this case was that
\[
d(G)\le c(A)+c(B).
\]
Indeed, if $A$ and $B$ have coprime orders, then this is proved, see \cite{MAN}. Also, if both
$A$ and $B$ are abelian, then $d(G)\le 2$ by It\^{o}'s result. This even holds for infinite
groups. In general however, the conjecture is not true. It was disproved first in \cite{CST}. 
Here several counter examples are given, for groups $G=AB$ with $(d(G),c(A),c(B))=(4,1,2),(5,2,2),(6,2,3)$.
There is even a nilpotent group $G$ of derived length $4$ which is the product of an
abelian group and a $2$-step solvable group. On the other hand, one might still be able to
bound $d(G)$ by a linear function in $c(A)$ and $c(B)$. \\[0.2cm]
It seems perhaps more natural, to bound $d(G)$ by a function depending on $d(A)$ and
$d(B)$. A first result here is due to L. Kazarin, see the reference in \cite{CST}:

\begin{prop}
Let $G=AB$ the product of two subgroups of coprime orders and assume that $G$ is solvable. Then
\[
d(G)\le 2d(A)d(B)+d(A)+d(B).
\]
\end{prop}

In some special cases better results are known. In \cite{COW} we find:

\begin{prop}
Suppose that $G=AB$ is a finite group with $A$ abelian and $B$ nilpotent and $2$-step
solvable. Then $G$ is solvable with $d(G)\le 4$.
\end{prop}

\section{Nildecomposable Lie algebras}

For nildecomposable Lie algebras over a field of characteristic zero there are similar 
questions as for finite groups, see \cite{BAT} for results and references. 
In particular it is not known, whether or not
there is a counterexample to the following estimate for finite-dimensional
nildecomposable Lie algebras $\Lg=\La+\Lb$ over a field of charactersitic zero,
\begin{align}\label{es}
d(\Lg) & \le c(\La)+c(\Lb).
\end{align}

In \cite{BAT} there is a counterexample given for infinite-dimensional Lie algebras.
The proof is inspired by the counterexamples for groups in \cite{CST}.
If one of the subalgebras, say $\Lb$, is an ideal, then the estimate \eqref{es} is correct.
Indeed, then $\Lg$ is an extension of $\Lb$ by $\Lg/\Lb$, so that $d(\Lg)\le d(\Lg/\Lb)+ d(\Lb)$.
Since $\Lg/\Lb=(\La+\Lb)/\Lb\simeq \La/(\La\cap \Lb)$ we obtain
\[
d(\Lg)\le d(\La)+d(\Lb)\le c(\La)+c(\Lb).
\]
If $\La$ and $\Lb$ are abelian, It\^{o}'s argument can also be applied to Lie algebras.
The result is again, that $\Lg$ is at most $2$-step solvable. This is well known, see for 
example \cite{PET9} and the references given therein. 
The next easiest case should be that $\La$ is abelian and $\Lb$ is of nilpotency class $2$.
There is indeed a result here, obtained in \cite{PET0}.

\begin{prop}
Let $\Lg=\La+\Lb$ be a Lie algebra over an arbitrary field of characteristic different from $2$
which is the sum of an abelian subalgebra $\La$ and a subalgebra $\Lb$ of nilpotency class $2$.
Then $d(\Lg)\le 10$.
\end{prop}

It appears to me that this bound is not yet optimal. Even in this case there is no counterexample 
known to \eqref{es}, i.e., to $d(\Lg) \le c(\La)+c(\Lb)=3$. Perhaps this is already the correct answer.

\end{document}